\documentclass[11pt, reqno]{amsart}
\usepackage{graphicx} % Required for inserting images
\usepackage{tikz}
\usetikzlibrary{calc}
\usetikzlibrary{cd}
\usetikzlibrary{matrix,fit,backgrounds}%On peut ajouter plus de bibliothèques tikz si besoin
\usepackage{mathtools}
\usepackage{mathdots}
\usepackage{tabularx}
\usepackage{array}
\usepackage{amssymb, amsfonts}
\usepackage{float}
\usepackage[normalem]{ulem}
\usepackage{cleveref}
\usepackage{enumitem}
\usepackage{caption}
\theoremstyle{plain}
\theoremstyle{definition}
\newtheorem{theorem}{Theorem}[section]

\newtheorem{remark}[theorem]{Remark}
\newtheorem{lemma}[theorem]{Lemma}
\newtheorem{definition}[theorem]{Definition}

\newtheorem{case}{Case}
\newtheorem{proposition}[theorem]{Proposition}

\newtheorem*{maintheorem*}{Main Theorem}

\newcommand{\N}{\mathbb{N}}

\title{Brick-finite skew-gentle algebras are representation-finite}
\author{Monica Garcia}
\address{LACIM, UQAM-Universit\'e Laval, {Montr\'eal}, {Canada}}
\email{garcia\_gallegos.monica\_del\_rocio@uqam.ca}

\author{Léa Lavoué}
\address{UVSQ, Universit\'e Paris-Saclay, Versailles, France}
\email{lea.lavoue@ens.uvsq.fr}
\begin{document}
\begin{abstract}
    We show that a skew-gentle algebra is brick-finite if and only if it is representation-finite, generalizing Plamondon’s original result for gentle algebras
\end{abstract}

\maketitle

\tableofcontents

\section{Introduction}
Bricks, or Schur representations, have long been an object of study in the representation theory of finite-dimensional algebras, where they appear as natural generalizations of simple modules \cite{gabriel1962des}. More recently, they have become central to several areas, notably in $\tau$-tilting theory, torsion theory, stability conditions, and the theory of $g$-vector fans. An algebra is said to be \emph{brick-finite} if it admits only finitely many bricks up to isomorphism. While every representation-finite algebra is clearly brick-finite, the converse is false in general and holds only under additional structural assumptions. This has led to a body on literature treating the distribution of bricks for a given algebra, and their implications. A summary of fundamental conjectures posed in recent years on brick-finiteness can be found in \cite{mousavand2025bricks}. 
\smallskip

An interesting example where brick-finiteness coincides with representation-finiteness is provided by \emph{gentle algebras}, a well-known class of string algebras introduced in \cite{assem2006iterated}. Using the combinatorial classification of their indecomposable modules, Plamondon proved that brick-finiteness implies representation-finiteness in this setting \cite{plamondon2019gentle}.
\smallskip

\emph{Skew-gentle} algebras (also referred as \emph{skewed}-gentle algebras) are a generalization of gentle algebras, which were originally introduced in \cite{Geiss1999skgentle}. 
While they retain much of the combinatorial flavor of gentle algebras, skew-gentle algebras exhibit new phenomena, notably the presence of special vertices and additional local structure. 
Their indecomposable modules were also completely classified when the base field is of characteristic different from two in \cite{CB1989}. 
As for gentle algebras, skew-gentle algebras and their indecomposable modules can also be studied through geometric models \cite{labardini2022derived, he2023geometric,AmiotBrustle2019}, specifically, through to surfaces with orbifolds. 
\smallskip

In this short note, we generalize Plamondon's result to skew-gentle algebras. Our main result is the following. 
\begin{maintheorem*}[\ref{thm:main}]
    Let $A \cong \Bbbk Q^{sp}/I^{sp}$ be a skew-gentle algebra over a field $\Bbbk$ of with $\operatorname{char}(\Bbbk) \neq 2$. Then, $A$ is brick-finite if and only if it is representation-finite.
\end{maintheorem*} 

While finalizing this manuscript, we were made aware of parallel work by Chang, Jin, Schroll, and Wang \cite{ChangJinSchrollWang2025}. In Section 3 of their paper, the authors investigate a closely related family of skew-gentle algebras to those considered in Section \ref{sec:min-bands} of this note. They show that these algebras are brick-infinite under the additional assumption that the base field is algebraically closed, from which a result analogous to Theorem \ref{thm:main} follows in this setting. They further study silting-discreteness of skew-gentle algebras via their geometric model.

%Test des commentaires. \MG{Voici un commentaire écrit par Monica}. \LL{Et une idée de commentaire pour Léa}.
\section*{Acknowledgements}
This work was carried out as part of Léa Lavoué's master's thesis \cite{Lea2025memoire} under the supervision of Thomas Brüstle and Monica Garcia. We would like to thank Thomas Brüstle for his wise advice, his availability, and his kind support throughout this project. MG thanks Kaveh Mousavand for his suggestions and ideas on minimal brick-infinite skew-gentle algebras. The authors are grateful for the financial support provided by the Canadian Mathematical Society for their participation in its 2025 summer meeting in Quebec City, where part of this work was presented as a poster \cite{Lea2025poster}. 
\section{Preliminaries on skew-gentled algebras}\label{Preliminaries}
The notion of skew-gentle algebra was originally introduced by de la Peña and Gei\ss\ in \cite{Geiss1999skgentle}. They are generalizations of \emph{gentle algebras}, in the sense that they arise as skewed versions of gentle algebras; see \cite{AmiotBrustle2019} for details. Since their apparition, many equivalent definitions have been used to study them, and for the purposes of this work, we will go back and forth between two of them. We first recall the definition of a gentle algebra.

\begin{definition}
    An algebra $A$ is said to be \emph{gentle} if it is isomorphic to a path algebra $\Bbbk Q/I$ where $Q = (Q_1, Q_0)$ is a finite quiver and $I$ is an admissible ideal, which together satisfy the following conditions.
    \begin{itemize}[label=--]
        \item Any vertex $v \in Q_0$ is the source (resp. the target) of at most two arrows.
        \item For any arrow $\alpha \in Q_1$, there exists at most one arrow $\beta$ and at most one arrow $\gamma$ such that $\beta \alpha \notin I$ and $\alpha \gamma \notin I$. 
        \item For any arrow $\alpha' \in Q_1$, there exists at most one arrow $\beta'$ and at most one arrow $\gamma'$ such that $\beta' \alpha' \in I$ and $\alpha' \gamma' \in I$. 
    \end{itemize}
\end{definition}

The following definition of skew-gentle algebra is drawn from \cite{AmiotBrustle2019}. 

\begin{definition}\label{def-skewgentle}
    An algebra $A$ is said to be skew-gentle if $A \cong \Bbbk Q^{sp}/I^{sp}$, where $Q^{sp} = (Q^{sp}_0, Q^{sp}_1)$ is a finite quiver and $I^{sp}$ is an ideal which together satisfy the following conditions:
\begin{enumerate} \setlength\itemsep{0.5em}
\item $S_p \subseteq Q_0$ is a subset of \emph{special vertices}  such that for every $i \in S_p$ there exists a \emph{special loop} $\varepsilon_i$.
\item $Q^{sp}_0 = Q_0$, $Q^{sp}_1 = Q_1 \sqcup \{ \varepsilon_i \mid i \in S_p \}$.
\item $I^{sp} = I \sqcup \{ \varepsilon_i^2 - e_i \mid i \in S_p \}$, where $I$ is an ideal generated by paths of length two.
\item $\Bbbk Q^{sp}/\langle I \sqcup \{ \varepsilon_i^2 \mid i \in S_p\} \rangle$ is a gentle algebra.
\end{enumerate}
\end{definition}
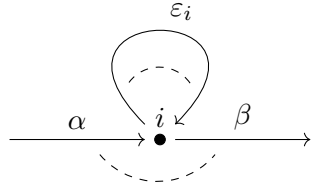
\begin{figure}[h!]
    \begin{tikzpicture}%[scale = 0.8,transform shape]
    \node[] (i) at (0,0) {\phantom{$i$}};
    \draw[fill=black] (0,0) circle (2pt);
    \node at (0, 0.3) {$i$};
    \draw[->] (-2,0) -- node[above] {$\alpha$} (i);
    \draw[->] (i) -- node[above] {$\beta$} (2,0);
    
    \draw[->] (i) to[out=135,in=45,looseness=15] node[above right] {$\varepsilon_i$} (i);

    \draw[dashed] (-.8,-0.2) arc (220:320:1);
    \draw[dashed] (0.4,0.75) arc (35:150:0.5);
\end{tikzpicture}

    \caption{Local diagram of the maximal configuration around a special vertex $i$ for a skew-gentle algebra.} \label{fig-gentle}
\end{figure}

The definition of a skew-gentle algebra can be reformulated in terms of its associated \emph{polarized quiver}, following \cite{Geiss2023}. This is the context in which the indecomposable modules of skew-gentle algebras where originally classified by Crawley-Boevey in \cite{CB1989}. For the sake of clarity and conciseness, we will abstain from introducing the complete definition of polarized quiver. The reader can find more details in \cite[Sections 2, 3 and 5]{Geiss2023}. 
\smallskip

For our purposes, it will also be useful to recall how a skew-gentle algebra can be described in terms of a quiver $\tilde{Q}$ and an admissible ideal $\tilde{I}$ (note that $I^{sp}$ is not admissible, since it is not contained in the two-sided ideal generated by all paths of length two of $\Bbbk Q^{sp}$). Let $Q^{sp} = (Q^{sp}_0, Q^{sp}_1)$ be as in Definition \ref{def-skewgentle}. For any $i \in Q^{sp}_0$, let $S(i) = \{+,-\}$ if $i \in S_p$ and $S(i) = \{o\}$ otherwise. A skew-gentle algebra $\Bbbk Q^{sp}/I^{sp}$ is isomorphic to the algebra $\Bbbk \tilde{Q}/\tilde{I}$, where:
\smallskip
\begin{gather*}
    \tilde{Q}_0 = \{ i_\rho \mid i \in Q^{sp}_0 \text{ and } \rho \in S(i)\} \\
    \vspace{2mm}
    \tilde{Q}_1 = \{ {}_{\tau}\alpha_{\sigma} \mid \alpha \in Q_1, \tau \in S(t(\alpha)),\ \sigma \in S(s(\alpha)) \} \\
    s({}_{\tau}\alpha_{\sigma}) = (s(\alpha), \sigma) \quad t({}_{\tau}\alpha_{\sigma}) = (t(\alpha), \tau). 
\end{gather*}

In other words, the quiver $\tilde{Q}$ is the result of duplicating each special vertex $i \in S_p$, and replacing each special loop with a family of at most 2 arrows (see \cref{fig:polarized-quiver} for an example).  Moreover, the ideal $\tilde{I}$ is induced by the relations given by $I$, while accounting for the duplication of arrows. Explicitly, we have:

$$\tilde{I} = \left\langle\displaystyle \bigcup_{\alpha\beta \in I} \left\{ \sum_{\sigma \in S(t(\beta))} {}_{\tau}\alpha_{\sigma}\ {}_{\sigma}\beta_{\rho} \mid \tau \in S(t(\alpha)),\ \rho \in S(s(\beta)) \right\}
 \right\rangle.$$ 
 
\begin{figure}[h!]
    \centering
    \begin{center}
\begin{tikzpicture}%[scale = 0.9,transform shape]
% Premier cas : i et j ordinaires
\node (i1) at (-1.5,0) {$\bullet$};
\node[above] at (i1) {$i$};
\node (j1) at (1.5,0) {$\bullet$};
\node[above] at (j1) {$j$};
\draw[->] (i1) -- node[above] {$\alpha$} (j1);
\node at (0,-2) {$i$ and $j$ ordinary vertices};

% Deuxième cas : i spécial, j ordinaire
\node (ip) at (5,0.8) {$\bullet$};
\node[above] at (ip) {$i_+$};
\node (im) at (5,-0.8) {$\bullet$};
\node[below] at (im) {$i_-$};
\node (j2) at (8,0) {$\bullet$};
\node[above] at (j2) {$j$};
\draw[->] (ip) -- node[above] {${}_{+}\alpha$} (j2);
\draw[->] (im) -- node[below] {${}_{-}\alpha$} (j2);
\node at (6.25,-2) {$i$ special and $j$ ordinary vertices};

% Troisième cas : i ordinaire, j spécial
\node (i2) at (-1.5,-4) {$\bullet$};
\node[above] at (i2) {$i$};
\node (jp) at (1.5,-3.3) {$\bullet$};
\node[above] at (jp) {$j_+$};
\node (jm) at (1.5,-4.8) {$\bullet$};
\node[below] at (jm) {$j_-$};
\draw[->] (i2) -- node[above] {$\alpha_+$} (jp);
\draw[->] (i2) -- node[below] {$\alpha_-$} (jm);
\node at (0,-6) {$i$ ordinary and $j$ special vertices};

% Quatrième cas : i et j spéciaux
\node (ip2) at (5,-3.3) {$\bullet$};
\node[above] at (ip2) {$i_+$};
\node (im2) at (5,-4.8) {$\bullet$};
\node[below] at (im2) {$i_-$};
\node (l1) at (5.5,-3.9) {${}_{+}\alpha_-$};
\node (jp2) at (8,-3.3) {$\bullet$};
\node[above] at (jp2) {$j_+$};
\node (l2) at (7.3,-3.9) {${}_{-}\alpha_+$};
\node (jm2) at (8,-4.8) {$\bullet$};
\node[below] at (jm2) {$j_-$};
\draw[->] (ip2) -- node[above] {${}_{+}\alpha_+$} (jp2);
\draw[->] (ip2) -- (jm2);
\draw[->] (im2) -- (jp2);
\draw[->] (im2) -- node[below] {${}_{-}\alpha_-$} (jm2);
\node at (6.5,-6) {$i$ and $j$ special vertices};

\end{tikzpicture}
\end{center}
    \caption{The different configurations in $\tilde{Q}$ associated to an arrow $\alpha : i \rightarrow j \in Q^{sp}_1$. To simplify notation, we omit the $o$ index associated to non-special vertices.}
    \label{fig:polarized-quiver}
\end{figure}
%\vspace{-1em}

\subsection{Indecomposable modules of skew-gentle algebras}
Recall that, as in the gentle case, the indecomposable modules of a skew-gentle algebra $A$ are classified by a subset of its strings and bands, specifically, those that are \emph{admissible}, which we denote by $\text{Adm}(Q^{sp})$ (see \cite[Definition 3.2]{Geiss2023}). Since the precise definition uses the notion of the associated polarized quiver, we skip it here. However, we recall that an admissible string $\omega = w_1 \cdots w_n$ of $Q^{sp}$ is said to be of type $(r,s) \in \{u, p\}^2$, with $r = p$ if and only if $t(w_1) \in S_p$ and $w_1$ is not a special loop, and $s = p$ if and only if $s(w_n) \in S_p$ and $w_n$ is not a special loop. 
\smallskip

For every $x \in \text{Adm}(Q^{sp})$, there is an associated finite-dimensional algebra $A_x$ as in the table below. Each algebra $A_x$ comes equipped with an automorphism $\iota$, which we also specify. 
{\renewcommand{\arraystretch}{1.8}
\begin{table}[H]
    \begin{tabular}{|c|c|c|}
        \hline
        Type of $x$ & $A_x$ & Autom. $\iota$\\ \hline
        $(u,u)$ & $\Bbbk$ & Id\\ \hline
        $(u,p)$ & $\Bbbk[T]/\langle T^2 -1\rangle$ & Id\\ \hline
        $(p,u)$ & $\Bbbk[T]/\langle T^2 -1\rangle$ & Id \\ \hline
        $(p,p)$ & $\frac{\Bbbk \langle T, S \rangle}{\langle T^2 -1, S^2-1 \rangle}$ & {$\substack{
            \iota(T) = S \\
            \iota(S) = T}$} \\ \hline
        band & $\Bbbk[T, T^{-1}]$ & $\iota(T) = T^{-1}$\\ \hline
    \end{tabular}
    \caption{\centering Algebras $A_x$ associated to the type of the admissible string or band $x \in \text{Adm}(Q^{sp})$.}
    \label{table:alg}
\end{table}}

% \begin{definition}
% We say that a string $v$ is admissible if the following two conditions are met.
%         \begin{enumerate}
%             \item $v$ is not symmetric: $v \neq v^{-1}$.
%             \item $v$ does not both start and end with an special arrow or its inverse: if $v = v_1\dots v_n$ then $v_1$ (or $v_1^{-1}$) $\notin Q_1^{sp}$ and $v_n$ (or $v_n^{-1}$) $\notin Q_1^{sp}$. \MG{what is $Q^{sp}_1$?}
%         \end{enumerate}
%     \end{definition}

%     \begin{definition}
%         A band $b$ is said admissible if it's primitive, that is, if there is no string $v$ such that $b = v^n$ for some $n  \geq 2$.
%     \end{definition}

For each $A_x$, denote by $\text{ind}(A_x)$ the set of isoclasses of indecomposable modules over $A_x$. Let $\Bbbk$ be a field with $\operatorname{char}(\Bbbk) \neq 2$. For a given skew-gentle algebra $\Bbbk Q^{sp}/I^{sp}$ define
$$\text{Adm}^{(\Bbbk)}(Q^{sp}) := \{(x, X) \mid x \in \text{Adm}(Q^{sp}) \text{ and } X \in \text{ind}(A_x)\}. $$
%The set $text{Adm}^{(\Bbbk)}(Q^{sp})$ is equipped with an equivalence relation $(x, X) \cong (x^{-1}, X^\iota)$
\begin{theorem}\cite{CB1989,BT-CB2024} \label{theo:class-indecomposable}
 Let $A = \Bbbk Q^{sp}/I^{sp}$ be a skew-gentle algebra over a field $\Bbbk$ of with $\operatorname{char}(\Bbbk) \neq 2$. Then there exists a surjective map
 $$\text{Adm}^{(\Bbbk)}(Q^{sp}) \to \operatorname{ind}(A)$$
 which associates to any $(x, X) \in \text{Adm}^{(\Bbbk)}(Q^{sp})$ a module $M_x(X)$, such that $M_x(X) \cong M_{x'}(X')$ if and only if $(x', X') = (x^{-1}, X^\iota)$, or, if $x$ is a band, if $x'$ is a rotation of $x$ and $X = X'$. 
\end{theorem}

\section{Minimal bands for skew-gentle algebras and proof of the main theorem}\label{sec:min-bands}
The goal is to prove the following theorem, which extends \cite[Theorem 1.1]{plamondon2019gentle} to skew-gentle algebras. First recall that a \emph{brick} is a module whose endomorphism ring is a division ring, that is, all of its non-zero endomorphisms are invertible. We will show the following. 
\begin{theorem}\label{thm:main}
    Let $A \cong \Bbbk Q^{sp}/I^{sp}$ be a skew-gentle algebra over a field $\Bbbk$ of with $\operatorname{char}(\Bbbk) \neq 2$. Then, $A$ is brick-finite if and only if it is representation-finite.
\end{theorem}

As in the gentle case, it is straightforward to see that being representation-finite is a sufficient condition for being brick-finite, since bricks are indecomposable modules. We will show that this is a sufficient condition.

\begin{lemma}
For any skew-gentle algebra $A = \Bbbk Q^{sp}/I^{sp}$ of infinite representation type, there exists a band in the gentle quiver $(Q^{sp},I')$, where $I' =\langle I \cup \{ \varepsilon_i^2 \mid i \in S_p\} \rangle$ as in Definition \ref{def-skewgentle} $(4)$.
\end{lemma}

\begin{proof}
As recalled in the previous section, the indecomposable modules of $A$ can be classified using admissible strings and bands. Note that Theorem \ref{theo:class-indecomposable} says that a given $x \in \text{Adm}(Q^{sp})$ gives rise to an infinite family of indecomposable modules only if $x$ is an admissible band, or if $x$ is an admissible string of type $(p,p)$. In the later case, by definition, $s(x)$ and $t(x)$ are special vertices, and thus $x^{-1} \varepsilon_{t(x)} x \ \varepsilon_{s(x)}$ is a band in $(Q^{sp},I')$. 
\smallskip

Thus, if $A$ is of infinite representation type, either $(Q^{sp},I')$ has a band, or $A$ must have infinitely many admissible strings, which in particular implies that $(Q^{sp},I')$ admits infinitely many strings. In particular, in the later case, $\Bbbk Q^{sp}/I'$ is a gentle algebra of infinite type, hence, it also admits a band as shown in \cite{butler1987auslander}. 
\end{proof}

As in \cite{plamondon2019gentle}, we will use the following lemmas to reduce the proof of \cref{thm:main} to a list of minimal cases.
    \begin{lemma}\cite[Theorem 5.12(d)]{DemonetIyamaReadingReitenThomas2019}, \cite[Theorem 1.4 ]{DemonetIyamaJasso2019}
        If $A$ is a brick-finite algebra and $J$ is an ideal of $A$, then $A/J$ is brick-finite.
    \end{lemma}

\begin{lemma}\cite[Lemma 4.1]{plamondon2019gentle} \label{lemma:min-bands}
    Let $A$ be a skew-gentle representation-infinite algebra. The quiver $(Q^{sp},I')$ associated with $A$ admits a minimal band $b$ such that:
\begin{enumerate}
    \item Either the band $b$ does not pass twice through the same vertex, except for the starting and ending vertices of $b$.
    \item Or the band $b$ is of the form $b= b'\omega b''\omega^{-1}$, where:
    \begin{enumerate}
        \item$b'$ and $b''$ are strings such that $s(b') = t(b')$ and  $s(b'') = t(b'')$.
        \item $b'^2$ and $b''^2$ are not strings.
        \item  $\omega$ may be trivial.
        \item None of the three strings $b', b''$ and $\omega$ pass twice through the same vertex, except for the starting and ending vertices of $b'$ and $b''$. In particular, the later are the only possible special vertices in $b$.
        \item The only possible common vertex between the three strings $b',\  b'',\  \omega$ are their starting and ending vertices.
    \end{enumerate}
\end{enumerate}
\end{lemma}

\begin{proof}
    This lemma is \cite[Lemma 4.1]{plamondon2019gentle} applied to $Q^{sp}$ in the case of skew-gentle algebras. The same proof applies because, by definition, $\Bbbk Q^{sp}/I'$ is gentle.
\end{proof}

We will consider the algebras $A/J$ for $J$ defined depending on the type  (1 or 2) of minimal band $b$ in the quiver $Q^{sp}$. In both cases, $J$ is the ideal of arrows and vertices which $b$ does not go through.

\begin{case}\textbf{$b$ is of type (1)}
    In this case, the algebra $A/J$ is isomorphic to the path algebra of $Q$ where the undirected graph of $Q$ is given by: 
 \begin{figure}[H]
    \centering
    \begin{tikzpicture}
    \node (M1) at (0,1) {$\bullet$};
    \node (M2) at (1,2) {$\bullet$};
    \node (M3) at (2,2) {$\bullet$};
    \node (M4) at (3,2) {};
    \node (M5) at (4,2) {};
    \node (M6) at (5,2) {$\bullet$};
    \node (M7) at (6,1) {$\bullet$};
    \node (M8) at (5,0) {$\bullet$};
    \node (M9) at (4,0) {};
    \node (M10) at (3,0) {};
    \node (M11) at (2,0) {$\bullet$};
    \node (M12) at (1,0) {$\bullet$};
    \draw[-] (M1) --  (M2);
    \draw[-] (M2) --  (M3);
    \draw[-] (M3) --  (M4);
    \draw[dashed, -] (M4) --  (M5);
    \draw[-] (M5) --  (M6);
    \draw[-] (M6) --  (M7);
    \draw[-] (M7) --  (M8);
    \draw[-] (M8) --  (M9);
    \draw[dashed,-] (M9) --  (M10);
    \draw[-] (M10) --  (M11);
    \draw[-] (M11) --  (M12);
    \draw[-] (M12) --  (M1);
\end{tikzpicture}
    \caption{Quiver of type $\widetilde{A_m}$ for $m\geq 1$.}
    \label{fig:An}
\end{figure}
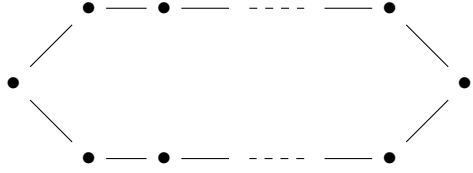
In this case, it is well known that the path algebra of this quiver is brick-infinite. A proof can be found in \cite{assem2006} and this case is also \cite[Example 3.1]{plamondon2019gentle}.
\end{case}
 
\begin{case}\textbf{$b$ is of type (2)}
In this case, the algebra $A/J$ is isomorphic to the skew-gentle algebra given by one of the following three quivers, depending on whether the vertices $s(b') = t(b')$ and $s(b'') = t(b'')$, which we mark with a star, are special:
\begin{enumerate}[label=(\roman*), leftmargin=*, labelindent=-1em, itemsep=1em]
    \item Both $s(b') = t(b')$ and $s(b'') = t(b'')$ are non-special: \vspace{0.25em}\\ 
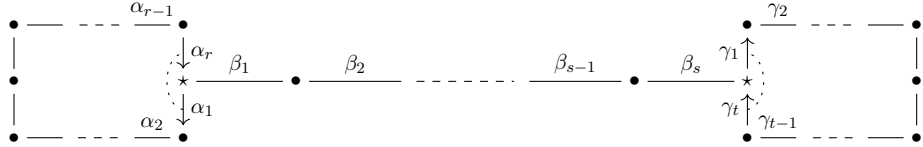
\begin{tikzpicture}[scale = 0.75,transform shape, baseline=(current bounding box.center)]    % sommets du rectangle gauche
    \node (A1) at (0,0) {$\bullet$};%coin BAS gauche
    \node (A2) at (1,0) {};
    \node (A3) at (2,0) {};
    \node (A4) at (3,0) {$\bullet$};%coin BAS droite
    \node (A5) at (3,1) {$\star$};%debut segment
    \node (A6) at (3,2) {$\bullet$};%coin HAUT droite
    \node (A7) at (2,2) {};
    \node (A8) at (1,2) {};
    \node (A9) at (0,2) {$\bullet$};%coin HAUT gauche
    \node (A10) at (0,1) {$\bullet$};
    % sommets segment
    \node (M2) at (5,1) {$\bullet$};
    \node (M3) at (7,1) {};
    \node (M4) at (9,1) {};
    \node (M5b) at (11,1) {$\bullet$};
    \node (M5) at (13,1) {$\star$};%fin du segment
    %sommets du rectangle de droite
    \node (B1) at (13,2) {$\bullet$};%coin haut gauche
    \node (B2) at (14,2) {};
    \node (B3) at (15,2) {};
    \node (B4) at (16,2) {$\bullet$};%coin haut droite
    \node (B5) at (16,1) {$\bullet$};
    \node (B6) at (16,0) {$\bullet$};%coin bas droite
    \node (B7) at (15,0) {};
    \node (B8) at (14,0) {};
    \node (B9) at (13,0) {$\bullet$};%coin bas gauche
    % trait du segment
    \draw[-] (A5) -- node[above] {$\beta_1$} (M2);
    \draw[-] (M2) -- node[above] {$\beta_2$} (M3);
    \draw[dashed,-] (M3) -- (M4);
    \draw[-] (M4) -- node[above] {$\beta_{s-1}$} (M5b);
    \draw[-] (M5) -- node[above] {$\beta_{s}$} (M5b);
    %cycle de gauche
    \draw[-] (A1) -- (A2);
    \draw[dashed,-] (A2) -- (A3);
    \draw[-] (A3) -- node[above] {$\alpha_2$} (A4);
    \draw[<-] (A4) -- node[right] {$\alpha_1$} (A5);
    \draw[<-] (A5) -- node[right] {$\alpha_r$} (A6);
    \draw[-] (A6) -- node[above] {$\alpha_{r-1}$} (A7);
    \draw[dashed,-] (A7) -- (A8);
    \draw[-] (A8) -- (A9);
    \draw[-] (A9) -- (A10);
    \draw[-] (A10) -- (A1);
    %cycle de droite
    \draw[-] (B1) --  node[above] {$\gamma_2$}(B2);
    \draw[dashed,-] (B2) -- (B3);
    \draw[-] (B3) --  (B4);
    \draw[-] (B4) --  (B5);
    \draw[-] (B5) -- (B6);
    \draw[-] (B6) -- (B7);
    \draw[dashed,-] (B7) -- (B8);
    \draw[-] (B8) -- node[above] {$\gamma_{t-1}$}(B9);
    \draw[->] (B9) -- node[left] {$\gamma_t$} (M5);
    \draw[->] (M5) -- node[left] {$\gamma_1$} (B1);
    %relations
    \draw[dash pattern=on 1pt off 3pt, bend left=80] ($(A4)!0.5!(A5)$) to ($(A5)!0.5!(A6)$);
    \draw[dash pattern=on 1pt off 3pt, bend right=80] ($(B9)!0.5!(M5)$) to ($(M5)!0.5!(B1)$);
\end{tikzpicture}
    \item One vertex is between $s(b') = t(b')$ and $s(b'') = t(b'')$ is special, and the other is non-special: \vspace{0.25em}\\ 
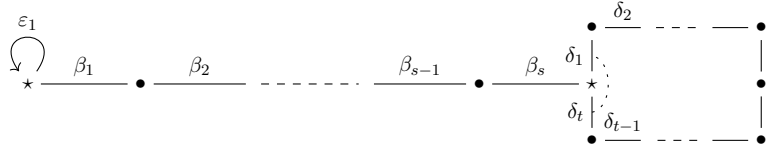
\begin{tikzpicture}[scale = 0.75,transform shape, baseline=(current bounding box.center)]
    \node (A1) at (0,0) {};
    \node (A5) at (3,1) {$\star$};%debut segment
    % sommets segment
    \node (M2) at (5,1) {$\bullet$};
    \node (M3) at (7,1) {};
    \node (M4) at (9,1) {};
    \node (M5b) at (11,1) {$\bullet$};
    \node (M5) at (13,1) {$\star$};%fin du segment
    %sommets du rectangle de droite
    \node (B1) at (13,2) {$\bullet$};%coin haut gauche
    \node (B2) at (14,2) {};
    \node (B3) at (15,2) {};
    \node (B4) at (16,2) {$\bullet$};%coin haut droite
    \node (B5) at (16,1) {$\bullet$};
    \node (B6) at (16,0) {$\bullet$};%coin bas droite
    \node (B7) at (15,0) {};
    \node (B8) at (14,0) {};
    \node (B9) at (13,0) {$\bullet$};%coin bas gauche
    % trait du segment
    \draw[-] (A5) -- node[above] {$\beta_1$} (M2);
    \draw[-] (M2) -- node[above] {$\beta_2$} (M3);
    \draw[dashed,-] (M3) -- (M4);
    \draw[-] (M4) -- node[above] {$\beta_{s-1}$} (M5b);
    \draw[-] (M5) -- node[above] {$\beta_{s}$} (M5b);
    %cycle de gauche
    \draw[->, looseness=8, in=125, out=55] (A5) to node[above] {$\varepsilon_1$}(A5);
    %cycle de droite
    \draw[-] (B1) --  node[above] {$\delta_2$}(B2);
    \draw[dashed,-] (B2) -- (B3);
    \draw[-] (B3) --  (B4);
    \draw[-] (B4) --  (B5);
    \draw[-] (B5) -- (B6);
    \draw[-] (B6) -- (B7);
    \draw[dashed,-] (B7) -- (B8);
    \draw[-] (B8) -- node[above] {$\delta_{t-1}$}(B9);
    \draw[-] (B9) -- node[left] {$\delta_t$} (M5);
    \draw[-] (M5) -- node[left] {$\delta_1$} (B1);
    %relations
    \draw[dash pattern=on 1pt off 3pt, bend right=80] ($(B9)!0.5!(M5)$) to ($(M5)!0.5!(B1)$);
\end{tikzpicture}
    \item Both $s(b') = t(b')$ and $s(b'') = t(b'')$ are special: \vspace{0.25em}\\ 
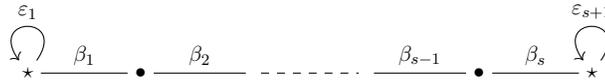
\begin{tikzpicture}[scale = 0.75,transform shape, baseline=(current bounding box.center)]
    \node (A1) at (0,0) {};
    \node (B6) at (16,0) {};
    \node (A5) at (3,1) {$\star$};%debut segment
    % sommets segment
    \node (M2) at (5,1) {$\bullet$};
    \node (M3) at (7,1) {};
    \node (M4) at (9,1) {};
    \node (M5b) at (11,1) {$\bullet$};
    \node (M5) at (13,1) {$\star$};%fin du segment
    \draw[-] (A5) -- node[above] {$\beta_1$} (M2);
    \draw[-] (M2) -- node[above] {$\beta_2$} (M3);
    \draw[dashed,-] (M3) -- (M4);
    \draw[-] (M4) -- node[above] {$\beta_{s-1}$} (M5b);
    \draw[-] (M5) -- node[above] {$\beta_{s}$} (M5b);
    \draw[->, looseness=8, in=125, out=55] (A5) to node[above] {$\varepsilon_1$}(A5);
    \draw[->, looseness=8, in=125, out=55] (M5) to node[above] {$\varepsilon_{s+1}$}(M5);
\end{tikzpicture}
\end{enumerate}
\end{case}

%\begin{remark}
%Next, we will give the construction of modules on the quiver with duplicate vertices, so we will look at the associated skew-gentle algebras $k\tilde{Q}/\tilde{I}$.
%\end{remark}
  
Case (i) has already been treated in \cite{plamondon2019gentle}. This is a gentle algebra which admits an infinite family of bricks, which are described in \cite[Proposition 3.3]{plamondon2019gentle}. For cases $(ii)$ and $(iii)$, we will use their realization in terms of the quiver $\tilde{Q}$ and the admissible ideal $\tilde{I}$ introduced in Section \ref{Preliminaries}. Recall that by Lemma \ref{lemma:min-bands} (2.b), we know that no vertex in $\omega$ is special, but for the ones adjacent to $b$ and $b'$. For case $(ii)$, when only one of the vertices in question is special, the associated quiver $\tilde{Q}$ is given by:

\begin{figure}[H]
    \centering
    \hspace{-4em}\makebox[\textwidth][c]{
    \begin{tikzpicture}[scale = 0.85,transform shape, baseline=(current bounding box.center)]
    \node (A1) at (0,1) {};
    \node (A5) at (3,0) {$\bullet$};
    \node[above] at (A5) {$1_-$};
    \node (A5') at (3,2) {$\bullet$};%debut segment
    \node[above] at (A5') {$1_+$};
    % sommets segment
    \node (M2) at (5,1) {$\bullet$};
    \node[above] at (M2) {$2$};
    \node (M3) at (7,1) {};
    \node (M4) at (9,1) {};
    \node (M4') at (11,1) {$\bullet$};
    \node[above] at (M4') {$s-1$};
    \node (M5) at (13,1) {$\bullet$};%fin du segment
    \node[right] at (M5) {$s$};
    %sommets du rectangle de droite
    \node (B1) at (13,2) {$\bullet$};%coin haut gauche
    \node[above] at (B1) {$s+1$};
    \node (B2) at (14,2) {};
    \node (B3) at (15,2) {};
    \node (B4) at (16,2) {$\bullet$};%coin haut droite
    \node (B5) at (16,1) {$\bullet$};
    \node (B6) at (16,0) {$\bullet$};%coin bas droite
    \node (B7) at (15,0) {};
    \node (B8) at (14,0) {};
    \node (B9) at (13,0) {$\bullet$};%coin bas gauche
    \node[below] at (B9) {$s+t-1$};
    % trait du segment
    \draw[-] (A5) -- node[above] {${}_-\beta_1$} (M2);
    \draw[-] (A5') -- node[above] {${}_+\beta_1$} (M2);
    \draw[-] (M2) -- node[above] {$\beta_2$} (M3);
    \draw[dashed,-] (M3) -- (M4);
    \draw[-] (M4) -- node[above, xshift = -0.3em] {$\beta_{s-1}$} (M4');
    \draw[-] (M4') -- node[above, xshift = 0.3em] {$\beta_s$} (M5);
    %cycle de droite
    \draw[-] (B1) --  node[below] {$\gamma_2$}(B2);
    \draw[dashed,-] (B2) -- (B3);
    \draw[-] (B3) --  (B4);
    \draw[-] (B4) --  (B5);
    \draw[-] (B5) -- (B6);
    \draw[-] (B6) -- (B7);
    \draw[dashed,-] (B7) -- (B8);
    \draw[-] (B8) -- node[above] {$\gamma_{t-1}$}(B9);
    \draw[->] (B9) -- node[left] {$\gamma_t$} (M5);
    \draw[->] (M5) -- node[left] {$\gamma_1$} (B1);
    %relations
    \draw[dash pattern=on 1pt off 3pt]
  ($(B9)!0.5!(M5)$)
  .. controls
    ($(B9)!0.5!(M5) + (0.55,0)$) and
    ($(M5)!0.5!(B1) + (0.55,0)$)
  ..
  ($(M5)!0.5!(B1)$);
    \end{tikzpicture}   }
    \caption{The $\tilde{Q}$ quiver associated to Case b $(ii)$. The direction of the arrows ${}_{-}\beta_1$ and ${}_{+}\beta_1$ follows that of $\beta_1$.}
    \label{fig:C2S}
\end{figure}
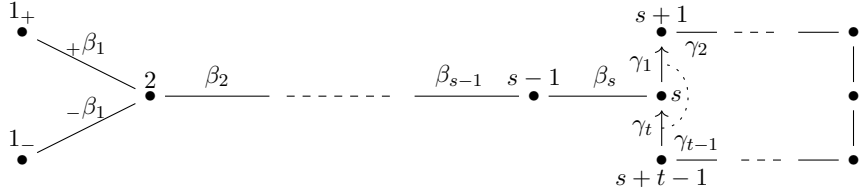
\vspace{-2.5em}

\begin{remark}
   If $t= 1$, we have that $\gamma_1$ is a non-special loop such that ${\gamma_1}^2 = 0$, and we must have that $s \geq 0$.
\end{remark}

\begin{proposition}
    The skew-gentle algebra given by the quiver $\tilde{Q}$ in Figure \ref{fig:C2S} has an infinite family of bricks.
\end{proposition}

\begin{proof}
Let $\tilde{Q}$ be the quiver in Figure \ref{fig:C2S}. We will only show this proposition for the case where $\beta_1$ is pointing towards the special vertex. For the other case, the infinite family of bricks will be given by the representations whose linear maps associated to ${}_{-}\beta_1$ and ${}_{+}\beta_1$\ are the transpositions of the maps we define here.
\smallskip

Let $n \in \N$ and consider the following representation $M_n$ of the quiver $\tilde{Q}$:

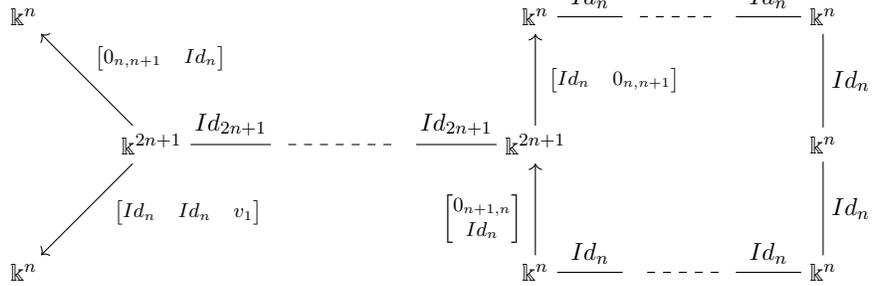
\vspace{-2em}
\begin{figure}[H]
    \centering
        \begin{tikzpicture}[scale = 0.85,transform shape]
    \node (A5) at (3,-1) {$\Bbbk^n$};
    \node (A5') at (3,3) {$\Bbbk^n$};%debut segment
    % sommets segment
    \node (M2) at (5,1) {$\Bbbk^{2n+1}$};
    \node (M3) at (7,1) {};
    \node (M4) at (9,1) {};
    \node (M5) at (11,1) {$\Bbbk^{2n+1}$};%fin du segment
    %sommets du rectangle de droite
    \node (B1) at (11,3) {$\Bbbk^n$};%coin haut gauche
    \node (B2) at (12.5,3) {};
    \node (B3) at (14,3) {};
    \node (B4) at (15.5,3) {$\Bbbk^n$};%coin haut droite
    \node (B5) at (15.5,1) {$\Bbbk^n$};
    \node (B6) at (15.5,-1) {$\Bbbk^n$};%coin bas droite
    \node (B7) at (14,-1) {};
    \node (B8) at (12.5,-1) {};
    \node (B9) at (11,-1) {$\Bbbk^n$};%coin bas gauche
    % trait du segment
    \draw[<-] (A5) -- node[pos=1.6, below=37pt]{\footnotesize$\begin{bmatrix}
        Id_n & Id_n & v_1\\
    \end{bmatrix}$} (M2);
    \draw[<-] (A5') -- node[pos=1.3, above=33pt]{\footnotesize$\begin{bmatrix}
        0_{n,n+1} & Id_n
    \end{bmatrix}$} (M2);
    \draw[-] (M2) -- node[above] {$Id_{2n+1}$} (M3);
    \draw[dashed,-] (M3) -- (M4);
    \draw[-] (M4) -- node[above] {$Id_{2n+1}$} (M5);
    %cycle de droite
    \draw[-] (B1) --  node[above] {$Id_n$}(B2);
    \draw[dashed,-] (B2) -- (B3);
    \draw[-] (B3) -- node[above] {$Id_n$} (B4);
    \draw[-] (B4) --  node[right] {$Id_n$}(B5);
    \draw[-] (B5) -- node[right] {$Id_n$}(B6);
    \draw[-] (B6) -- node[above] {$Id_n$}(B7);
    \draw[dashed,-] (B7) -- (B8);
    \draw[-] (B8) -- node[above] {$Id_n$}(B9);
    \draw[->] (B9) -- node[pos=0.4, left=2pt] {\footnotesize$\begin{bmatrix}
        0_{n+1,n}\\Id_n
    \end{bmatrix}$} (M5);
    \draw[->] (M5) -- node[midway, right=2pt] {\footnotesize$\begin{bmatrix}
        Id_n &0_{n,n+1}
    \end{bmatrix}$} (B1);
    \end{tikzpicture}  
    %\caption{}
    \label{fig:C2g}
\end{figure}
\vspace{-2em}

\noindent where $v_1$ is the vector of the standard basis of $\Bbbk^n$ with a $1$ in its first entry. We are going to show that $\text{End}(M_n) \cong \Bbbk$, which implies that $M_n$ is a brick. Let $g\in \text{End}(M_n) \subseteq \prod_{i\in \tilde{Q}_0} \text{Mat}_{d_i\times d_i}(\Bbbk)$. 
We write 
$$g = (\underbrace{g_{1_+}, g_{1_-}}_{\varepsilon_1},\underbrace{g_2, \dots, g_s}_{\omega},\underbrace{g_{s+1},\dots,g_{s+t-1}}_{b''}),$$ 
\noindent where \( g_{1_+} \) and \( g_{1_-} \) denote the endomorphisms of the vector spaces at the split vertices \( 1_+ \) and \( 1_- \); \( g_2,\dots,g_s \) those associated with the vertices in \( \omega \); and \( g_{s+1},\dots,g_{s+t-1} \) those associated with the string \( \gamma_2 \cdots \gamma_{t-1} \).

Since $g$ is a map between representations, the following relations must be satisfied: 

\begin{gather}
    \begin{bmatrix}
        0_{n,n+1}&Id_n
    \end{bmatrix}g_2 = g_{1_+}\begin{bmatrix}
        0_{n,n+1}&Id_n
    \end{bmatrix} \\
    \begin{bmatrix}
        Id_n & Id_n & v_1
    \end{bmatrix}g_2 = g_{1_-}\begin{bmatrix}
        Id_n & Id_n & v_1
    \end{bmatrix} \\
    g_2 = \dots = g_s\quad  \\
    \begin{bmatrix}
        Id_n & 0_{n,n+1}
    \end{bmatrix}g_s = g_{s+1}\begin{bmatrix}
        Id_n & 0_{n,n+1}
    \end{bmatrix} \\
    \begin{bmatrix}
        0_{n+1,n} \\ Id_n
    \end{bmatrix} g_{s+t-1} = g_{s} \begin{bmatrix}
        0_{n+1,n} \\ Id_n
    \end{bmatrix} \\
    g_{s+1}= \dots = g_{s+t-1} 
\end{gather}

We let \vspace{-1em} \begin{gather*}
(\beta_{i,j})_{1\leq i,j \leq n} = g_{1_+}, \qquad (\gamma_{i,j})_{1\leq i,j \leq n} = g_{1_-}, \\
(\delta_{i,j})_{1\leq i,j \leq 2n+1} = g_{2}= \dots =g_s, \quad \text{and} \\
(\alpha_{i,j})_{1\leq i,j \leq n} = g_{s+1}=\dots= g_{s+t-1}.
\end{gather*}
    
\noindent By relation $(4)$, we have that $(\delta_{i,j})_{1\leq i \leq n, \ 1 \leq j \leq 2n+1}
    = \begin{bmatrix}
          g_{s+1} & 0_{n,n+1} 
    \end{bmatrix} $, that is:
    $$\begin{bmatrix}
        \delta_{1,1} & \dots & \delta_{1,n}& \delta_{1,n+1} &\dots & \delta_{1,2n+1}\\
        \vdots &&\vdots&\vdots&& \vdots\\
        \delta_{n,1} & \dots & \delta_{n,n}& \delta_{n,n+1} &\dots & \delta_{n,2n+1}
    \end{bmatrix}=\begin{bmatrix}
        \alpha_{1,1} & \dots & \alpha_{1,n}& 0 &\dots & 0\\
        \vdots &&\vdots&\vdots&& \vdots\\
        \alpha_{n,1} & \dots & \alpha_{n,n}& 0 &\dots & 0
    \end{bmatrix}, $$
    which implies that $\alpha_{i,j} = \delta_{i,j}$ for all $ 1\leq i,j \leq n$ and $\delta_{i,j} = 0$ for all $1\leq i \leq n$ and $ n+1\leq j \leq 2n+1$.

    Then we look at relation $(5)$, which implies that $\begin{bmatrix} 0_{n+1,n} \\ g_{s+t-1}
    \end{bmatrix} = (\delta_{i,j})_{\hspace{-0.5em}\substack{1\leq i \leq 2n+1\\ \ n+2 \leq j \leq 2n+1}}$, or:
    $$\begin{bmatrix}
        0&\dots&0\\
        \vdots &&\vdots\\
        0&\dots&0\\
        0&\dots&0\\
        \alpha_{1,1}&\dots&\alpha_{1,n}\\
        \vdots &&\vdots\\
        \alpha_{n,1}&\dots&\alpha_{n,n}\\
    \end{bmatrix} = \begin{bmatrix}
        \delta_{1,n+2}&\dots&\delta_{1,2n+1}\\
        \vdots &&\vdots\\
        \delta_{n,n+2}&\dots&\delta_{n,2n+1}\\
        \delta_{n+1,n+2}&\dots&\delta_{n+1,2n+1}\\
        \delta_{n+2,n+2}&\dots&\delta_{n+2,2n+1}\\
        \vdots &&\vdots\\
        \delta_{2n+1,n+2}&\dots&\delta_{2n+1,2n+1}\\
    \end{bmatrix}.$$

    We obtain that $\delta_{i,j} = 0$ for all $1\leq i \leq n+1$ and $n+2 \leq j \leq 2n+1$, and that $\alpha_{i,j} = \delta_{i+n+1,j+n+1}$ for all $ 1 \leq i,j \leq n$. Thanks to the previous relations, we get that $$g_s = \begin{bmatrix}
        \alpha_{1,1}&\dots&\alpha_{1,n}&0&0&\dots&0\\
        \vdots&&\vdots&\vdots&\vdots&&\vdots\\
        \alpha_{n,1}& \dots& \alpha_{n,n}&0&0&\dots&0\\
        \delta_{n+1,1}&\dots & \delta_{n+1,n}&\delta_{n+1,n+1}&0&\dots&0\\
        \delta_{n+2,1}&\dots&\delta_{n+2,n}&\delta_{n+2,n+1}& \alpha_{1,1}&\dots&\alpha_{1,n}\\
        \vdots&&\vdots&\vdots&\vdots&&\vdots\\
        \delta_{2n+1,1}&\dots&\delta_{2n+1,n}&\delta_{2n+1,n+1}& \alpha_{n,1}&\dots&\alpha_{n,n}\\
    \end{bmatrix}.$$
    Now, by using relation $(1)$, we also have that $
        (\delta_{i,j})_{\substack{n+2\leq i \leq 2n+1 \\ \ 1 \leq j \leq 2n+1}}
     =\begin{bmatrix}
        0_{n,n+1}&g_{1_+}
    \end{bmatrix}$: \\
    $${\setlength{\arraycolsep}{1.5pt}\begin{bmatrix}
        \delta_{n+2,1}&\dots&\delta_{n+2,n}&\delta_{n+2,n+1}& \alpha_{1,1}&\dots&\alpha_{1,n}\\
        \vdots&&\vdots&\vdots&\vdots&&\vdots\\
        \delta_{2n+1,1}&\dots&\delta_{2n+1,n}&\delta_{2n+1,n+1}& \alpha_{n,1}&\dots&\alpha_{n,n}\\
    \end{bmatrix}} = {\setlength{\arraycolsep}{2pt} \begin{bmatrix}
        0&\dots&0&0& \beta_{1,1}&\dots&\beta_{1,n}\\
        \vdots&&\vdots&\vdots&\vdots&&\vdots\\
        0&\dots&0&0& \beta_{n,1}&\dots&\beta_{n,n}\\
    \end{bmatrix}},$$
    {\noindent which implies that $\delta_{i,j} = 0$ for all $n+2 \leq i \leq 2n+1$ and $ 1 \leq j \leq n+1$ and that $g_{1_+} = g_{s+1}$. Finally, by $(2)$, we have that $\begin{bmatrix}
        Id_n & Id_n & v_1
    \end{bmatrix}g_2 = g_{1_-}\begin{bmatrix}
        Id_n & Id_n & v_1
    \end{bmatrix}$, which gives the following: $$\begin{tikzpicture}[baseline=(m.center)]
  \matrix (m) [matrix of math nodes,
    left delimiter={[},right delimiter={]},
    nodes={anchor=center,minimum width=1.5em,minimum height=2em},
    column sep=0pt]
  {
    \alpha_{1,1}+\delta_{n+1,1} & \dots & \alpha_{1,n}+\delta_{n+1,n} & \delta_{n+1,n+1} & \alpha_{n,1} & \dots & \alpha_{n,n-1} & \alpha_{n,n} \\
    \alpha_{2,1} & \dots & \alpha_{2,n} & 0 & \alpha_{1,1} & \dots & \alpha_{1,n-1} & \alpha_{1,n} \\
    &&&& \alpha_{2,1} & \dots & \alpha_{2,n-1} & \\
    \vdots & & \vdots & \vdots & \vdots & & \vdots & \vdots \\
    \alpha_{n,1} & \dots & \alpha_{n,n} & 0 & \alpha_{n-1,1} & \dots & \alpha_{n-1,n-1} & \alpha_{n-1,n} \\
  };
  \begin{scope}[on background layer]
    \node[fill=green!20,rounded corners,inner sep=0pt,fit=(m-2-1) (m-5-3)] {};
    \node[fill=gray!20,rounded corners,inner sep=0pt,fit=(m-2-4) (m-5-4)] {};
    \node[fill=yellow!30,rounded corners,inner sep=0pt,fit=(m-2-5) (m-5-7)] {};
    \node[fill=orange!30,rounded corners,inner sep=0pt,fit=(m-1-5) (m-1-7)] {};
    \node[fill=blue!20,rounded corners,inner sep=0pt,fit=(m-2-8) (m-5-8)] {};
    \node[fill=magenta!20,rounded corners,inner sep=0pt,fit=(m-1-8) (m-1-8)] {};
  \end{scope}
\end{tikzpicture}=$$  $$\begin{tikzpicture}[baseline=(m.center)]
  \matrix (m) [matrix of math nodes,
    left delimiter={[},right delimiter={]},
    nodes={anchor=center,minimum width=1.5em,minimum height=2em},
    column sep=0pt]
  {
    \gamma_{1,1}&& \dots&\gamma_{1,n}& \gamma_{1,1}& \gamma_{1,2}&\dots& \gamma_{1,n}&\gamma_{1,1}\\ \gamma_{2,1}&\gamma_{2,2}&\dots&\gamma_{2,n}&\gamma_{2,1}&\gamma_{2,2}&\dots&\gamma_{2,n}&\gamma_{2,1}\\
        \vdots&\vdots&&\vdots&\vdots&\vdots&&\vdots&\vdots\\
        \gamma_{n,1}&\gamma_{n,2}&\dots&\gamma_{n,n}&\gamma_{n,1}&\gamma_{n,2}&\dots&\gamma_{n,n}&\gamma_{n,1}\\
  };
  \begin{scope}[on background layer]
    \node[fill=green!20,rounded corners,inner sep=0pt,fit=(m-2-1) (m-4-4)] {};
    \node[fill=gray!20,rounded corners,inner sep=0pt,fit=(m-2-5) (m-4-5)] {};
    \node[fill=yellow!30,rounded corners,inner sep=0pt,fit=(m-2-6) (m-4-8)] {};
    \node[fill=orange!30,rounded corners,inner sep=0pt,fit=(m-1-6) (m-1-8)] {};
    \node[fill=blue!20,rounded corners,inner sep=0pt,fit=(m-2-9) (m-4-9)] {};
    \node[fill=magenta!20,rounded corners,inner sep=0pt,fit=(m-1-9) (m-1-9)] {};
  \end{scope} \end{tikzpicture}.$$
    \noindent By the relations given by the gray region, we get that $\gamma_{i,1}= 0$ for all $2 \leq i \leq n$.
    In particular, we have $\alpha_{i,1} = \gamma_{i,1} = 0$ for all $2 \leq i \leq n$, which is implied by the first column of the green region. But by the relations given together by the green and yellow regions, we get that $\alpha_{i,j} = \alpha_{i-1,j-1} $ for all $2 \leq i,j \leq n$; thus we can conclude that $\alpha_{i,j} = 0$ for all $1 \leq j <i \leq n$. 
    We also have that $\gamma_{i,1} = \alpha_{i-1,n} $ for all $2 \leq i\leq n$, implying that $\alpha_{i-1,n} = 0$, due to the relations given by the blue region. Together with the fact that $\alpha_{i,j} = \alpha_{i-1,j-1}$ for all $2 \leq i,j \leq n$, we can conclude $\alpha_{i,j} = 0$ for all $1 \leq i < j \leq n$.}
    \smallskip

    Moreover, by the relations given by the orange region, we get that $\alpha_{n,i-1} = \gamma_{1,i} $ for all  $2 \leq i\leq n$. We also have $\gamma_{1,i} = \alpha_{1,i} + \delta_{n+1,i} $ for all $1 \leq i\leq n$, but we already know that $\gamma_{1,i}= \alpha_{1,i} = 0$ for these $2 \leq i\leq n$, so $\delta_{n+1,i} = 0$ for all  $2 \leq i\leq n$. Lastly, the pink region implies that $\alpha_{n,n} = \gamma_{1,1} = \alpha_{1,1} + \delta_{n+1, 1}$ and $ \gamma_{1,1} = \delta_{n+1, n+1}$. But as we have seen $\alpha_{n,n} = \alpha_{1,1}$; which implies that $\alpha_{i,i} = \gamma_{i,i} = \delta_{n+1,n+1}$ for all $1 \leq i\leq n $ and that $\delta_{n+1,i} = 0$ for all  $1 \leq i\leq n$.
    We conclude that $g_{1_+}=g_{1_-}=g_{s+1}=\dots= g_{s+t-1} = k \cdot  Id_n$ and that $g_2 =\dots = g_s = k \cdot  Id_{2n+1} $ for some $k \in \Bbbk$, which implies that $\text{End}(M_n) \cong \Bbbk$.
\end{proof}

The only case left is when both vertices are special. In this case, it is necessary that $s \geq 1$, since it is impossible to have two special loops at the same vertex. The associated quiver $\tilde{Q}$ is the following.
%\begin{figure}[H]
%    \centering
%    \input{C3}
%    \caption{third possible quiver.}
%    \label{fig:C3}
%\end{figure}
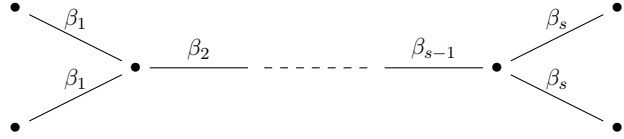
\begin{figure}[H]
    \centering
    \begin{tikzpicture}[scale = 0.8,transform shape]
    \node (A1) at (0,0) {};
    \node (B6) at (16,0) {};
    \node (M1) at (3,1) {$\bullet$};
    \node (M1') at (3,-1) {$\bullet$};
    \node (M2) at (5,0) {$\bullet$};
    \node (M3) at (7,0) {};
    \node (M4) at (9,0) {};
    \node (M5) at (11,0) {$\bullet$};
    \node (M6) at (13,1) {$\bullet$};
    \node (M6') at (13,-1) {$\bullet$};
    \draw[-] (M1) -- node[above] {$\beta_1$} (M2);
    \draw[-] (M1') -- node[above] {$\beta_1$} (M2);
    \draw[-] (M2) -- node[above] {$\beta_2$} (M3);
    \draw[dashed,-] (M3) -- (M4);
    \draw[-] (M4) -- node[above] {$\beta_{s-1}$} (M5);
    \draw[-] (M5) -- node[above] {$\beta_s$} (M6);
    \draw[-] (M5) -- node[above] {$\beta_s$} (M6');
\end{tikzpicture}
    \caption{The quiver $\tilde{Q}$ associated to case 2 $(iii)$.}
    \label{fig:C3S}
\end{figure}

\begin{proposition}
    The quiver $\tilde{Q}$ in Figure \ref{fig:C3S} has a family of infinitely-many bricks.
\end{proposition}
\begin{proof}
    When $s > 1$, the quiver $\tilde{Q}$ is of type $\tilde{D}_{s+2}$ which is known to be brick-infinite. Indeed, it is a classical result that, since $\tilde{D}_{s+2}$ is a hereditary tame algebra, its preprojective component has infinitely many indecomposable modules, and that these are bricks \cite{ARS1995}. If $s=1$, the quiver $Q^{sp}$ and the corresponding quiver $\tilde{Q}$ are as follows:
    
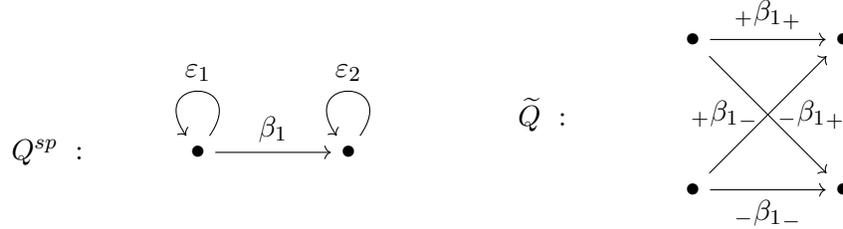
\vspace{-1.5em}
    \begin{figure}[H]
        \centering 
        \begin{tikzpicture}[baseline=(current bounding box.center)]
    \node (M) at (-2,0){$Q^{sp} \ :$};
    \node (M1) at (0,0) {$\bullet$};
    \node (M6) at (2,0) {$\bullet$};
    \draw[->] (M1) -- node[above] {$\beta_1$} (M6);
    \draw[->, looseness=8, in=125, out=55] (M1) to node[above] {$\varepsilon_{1}$}(M1);
    \draw[->, looseness=8, in=125, out=55] (M6) to node[above] {$\varepsilon_{2}$}(M6); \end{tikzpicture} \hspace{3em}
    \begin{tikzpicture}[baseline=(current bounding box.center)]
    \node (N1) at (6,0) {$\widetilde{Q} \ :$};
    \node (N2) at (8,-1) {$\bullet$};
    \node (N3) at (8,1) {$\bullet$};
    \node (N4) at (10,-1) {$\bullet$};
    \node (N5) at (10,1) {$\bullet$};
    \draw[->] (N2) -- node[below] {${}_-{\beta_1}_-$} (N4);
    \draw[->] (N3) -- node[above] {${}_+{\beta_1}_+$} (N5);
    \draw[->] (N3) -- node[right] {${}_-{\beta_1}_+$} (N4);
    \draw[->] (N2) -- node[left] {${}_+{\beta_1}_-$} (N5);
\end{tikzpicture}
    \end{figure}
\vspace{-1em}

In this case, an easy computation gives that the following representation $M_n$ give us an infinite family of bricks:
\begin{center}
    \begin{tikzcd}[ampersand replacement=\&, column sep = large, row sep = huge]
        \Bbbk^n \arrow[rr, "Id_n"] \arrow[rrd, "{\scriptstyle\begin{bmatrix} 0 \\ Id_n \end{bmatrix}}" {xshift = 0.95cm, yshift = -15pt}] \& \& \Bbbk^n \\
        \Bbbk^n \arrow[rru, crossing over, "Id_n" {xshift = -15pt, yshift = -8pt}] \arrow[rr, "{\scriptstyle\begin{bmatrix} Id_n \\ 0 \end{bmatrix}}"'] \& \& \Bbbk^{n+1}
    \end{tikzcd}
\end{center}
\end{proof}

\bibliographystyle{alpha}
\bibliography{biblio-skewgentle.bib}
\end{document}